\documentclass[12pt,a4paper,pdfps]{article}
\usepackage[cp1251]{inputenc}
\usepackage[english]{babel}
\usepackage{color}
\usepackage[colorlinks,linkcolor=blue,filecolor=blue,citecolor=blue]{hyperref}
\usepackage{amsmath,amssymb,amsthm}
\usepackage{graphicx}
\usepackage{comment}
\usepackage{enumitem}
\usepackage{subfigure,multirow}
\usepackage[left=30mm, top=20mm, right=20mm, bottom=20mm, nohead, footskip=7mm]{geometry}
\usepackage{setspace}
\DeclareMathOperator{\Isom}{Isom}
\DeclareMathOperator{\Ker}{Ker}

\DeclareMathOperator{\ad}{\mathrm{ad}}
\DeclareMathOperator{\Ad}{\mathrm{Ad}}

\DeclareMathOperator{\sspan}{\mathrm{span}}

\newcommand{\g}{\mathfrak{g}}

\newcommand{\kk}{\mathfrak{k}}
\newcommand{\ii}{\mathfrak{i}}
\newcommand{\m}{\mathfrak{m}}
\newcommand{\rr}{\mathfrak{r}}
\newcommand{\pp}{\mathfrak{p}}
\newcommand{\n}{\mathfrak{n}}

\newcommand{\K}{\mathcal{K}}
\newcommand{\B}{\mathcal{B}}

\newcommand{\R}{\mathbb{R}}
\newcommand{\Z}{\mathbb{Z}}

\newcommand{\PSL}{\mathrm{PSL}}
\newcommand{\SO}{\mathrm{SO}}

\newcommand{\so}{\mathfrak{so}}
\newcommand{\Hl}{\vec{H}}

\newcommand{\Hv}{\Hl_{\mathrm{vert}}}

\theoremstyle{definition}
\newtheorem{definition}{Definition}
\newtheorem{remark}{Remark}
\newtheorem{example}{Example}

\theoremstyle{plain}
\newtheorem{corollary}{Corollary}
\newtheorem{lemma}{Lemma}
\newtheorem{theorem}{Theorem}
\newtheorem{proposition}{Proposition}

\newenvironment{enumerate*}%
  {\begin{enumerate}%
    \setlength{\itemsep}{1pt}%
    \setlength{\parskip}{1pt}}%
  {\end{enumerate}}

\title{Homogeneous geodesics in sub-Riemannian geometry
\footnote{The work is supported by the Russian Science Foundation under grant 22-21-00877 (https://rscf.ru/en/project/22-21-00877/) and performed in Ailamazyan Program Systems Institute of Russian Academy of Sciences.}}

\author{
A.\,V.~Podobryaev \\ A.\,K.~Ailamazyan Program Systems
Institute of RAS \\ \tt{alex@alex.botik.ru} \\
}

\date{}

\begin{document}

\maketitle

\begin{abstract}
We study homogeneous geodesics of sub-Riemannian manifolds, i.e.,
normal geodesics that are orbits of one-parametric subgroups of isometries.
We obtain a criterion for a geodesic to be homogeneous in terms of its initial momentum.
We prove that any weakly commutative sub-Riemannian homogeneous space is geodesic orbit,
that means all geodesics are homogeneous.
We discuss some examples of geodesic orbit sub-Riemannian manifolds.
In particular, we show that geodesic orbit Carnot groups are only groups of step $1$ and $2$.
Finally, we get a broad condition for existence of at least one homogeneous geodesic.


\end{abstract}

\section*{\label{sec-introduction}Introduction}

In this paper by \emph{a sub-Riemannian manifold} we mean a triple $(M, \Delta, \B)$, where
$M$ is a smooth manifold,
$\Delta$ is \emph{a smooth distribution of constant rank} on $M$, i.e.,
a sub-bundle of the tangent bundle $TM$,
and $\B$ is a scalar product on $\Delta$ smoothly depending on a point of the manifold $M$.
We assume that the distribution $\Delta$ is bracket generating, i.e.,
a finite number of Lie brackets of the vector fields that are tangent to the distribution span the whole tangent bundle $TM$.

\begin{definition}
\label{def-sR}
\emph{An admissible (or a horizontal) curve} is a Lipschitz curve $\gamma : [0,T] \rightarrow M$ such that
$\dot{\gamma}(t) \in \Delta_{\gamma(t)}$ for a.e. $t \in [0, T]$.
\emph{The sub-Riemannian length} of the admissible curve $\gamma$ equals
$$
\int_{0}^{T}{\sqrt{\B(\dot{\gamma}(t),\dot{\gamma}(t))} \, dt}.
$$
\emph{The sub-Riemannian distance} (or \emph{the Carnot-Caratheodory distance}) between points $p,q \in M$ is
the infimum of sub-Riemannian length of admissible curves connecting the points $p$ and $q$.
A curve with natural parametrization is called \emph{a geodesic} if its sufficiently small arcs are shortest arcs.
\end{definition}

Thus, the sub-Riemannian structure on the manifold $M$ provides a structure of metric space on $M$
(for details and for sub-Riemannian geometry in general we refer to the book~\cite{agrachev-barilari-boscain}).

\begin{definition}
\label{def-isom}
A diffeomorphism $f : M \rightarrow M$ of a sub-Riemannian manifold $M$ is called \emph{an isometry} if
$$
d_mf(\Delta_m) = \Delta_{f(m)}, \quad \B_m(v, w) = \B_{f(m)}(d_mf(v), d_mf(w)) \quad
\text{for any}\quad  m \in M, \ v, w \in \Delta_m.
$$
\end{definition}

One can define an isometry more directly as a distance preserving homeomorphism.
This definition is equivalent to Definition~\ref{def-isom} for an equiregular sub-Riemannian structure,
in particular, for a left-invariant sub-Riemannian structure.
Moreover, in this case the group of isometries is a finite-dimensional Lie group~\cite{capogna-ledonne}.

A sub-Riemannian geodesic is called \emph{homogeneous} if it is an orbit of a one-parametric group of isometries.
A sub-Riemannian manifold is called \emph{geodesic orbit} if any normal geodesic is homogeneous.
Firstly this notion appears in the paper~\cite{kowalski-vanhecke} by O.~Kowalski and L.~Vanhecke in Riemannian case.
We refer to the book by V.\,N.~Berestovskii and Yu.\,G.~Nikonorov~\cite{berestovskii-nikonorov-book} for a historical overview of the results obtained for homogeneous geodesics and geodesic orbit Riemannian manifolds.

The objective of this paper is a generalization to sub-Riemannian case of some of the results on homogeneous geodesics and the geodesic orbit property known for Riemannian manifolds.
We use an approach from geometric control theory~\cite{agrachev-sachkov} and
characterize normal geodesics as projections of phase curves of a Hamiltonian vector field on the cotangent bundle.
Any geodesic starting at a fixed point is determined by an initial covector (momentum) instead of an initial vector in Riemannian geometry.
Therefore, we study homogeneous properties of geodesics in terms of initial momenta.

The paper has the following structure.
In Section~\ref{sec-geodesics} we consider a sub-Riemannian problem as an optimal control problem,
then we define corresponding extremal curves in terms of a Hamiltonian vector field on the cotangent bundle.
Then in Section~\ref{sec-homogeneousgeodesics} we obtain a criterion for homogeneous geodesics using vertical part of this Hamiltonian system and consider some examples.
Especially we discuss geodesics of left-invariant sub-Riemannian structures on Carnot groups in Section~\ref{sec-nilpotent}.
We prove there that only $2$-step Carnot groups are geodesic orbit.
Section~\ref{sec-go} is devoted to several general properties of sub-Riemannian geodesic orbit manifolds.
In particular, we get a criterion for a sub-Riemannian manifold to be geodesic orbit,
we prove that the geodesic flow of a geodesic orbit sub-Riemannian manifold is integrable in non-commutative sense.
Also we obtain that any weakly commutative (in particular, weakly symmetric) sub-Riemannian homogeneous space is geodesic orbit.
Finally, in Section~\ref{sec-existence} we obtain some broad conditions for existence of at least one homogeneous geodesic on a sub-Riemannian manifold.

Let us introduce some notation that we will use throughout the paper.
We denote Lie groups by uppercase Latin letters and corresponding Lie algebras by same lowercase Gothic letters.
The Killing form of a Lie algebra $\g$ is $\K_{\g}$.
If it is clear which Lie group we are speaking about, then we omit a subscript of this symbol.
The left and right shifts by an element $g$ of a Lie group are denoted by $L_g$ and $R_g$ respectively.
We use the same letters for their differentials. The identity element of a group is denoted by symbol $e$.
Since we consider a left-invariant distribution we will write just $\Delta$ for the subspace $\Delta_e \subset T_eM$.
If $U \subset V$ is a vector subspace of a vector space $V$, then $U^{\circ} \subset V^*$ is its annihilator in the dual space.
For a smooth function $F$ on a cotangent bundle $T^*M$ we denote by $\vec{F}$ its symplectic gradient with respect to the canonical symplectic structure on $T^*M$, i.e.,
the Hamiltonian vector field corresponding to the Hamiltonian $F$.

In this paper we will consider a Lie group of isometries $G \subset \Isom{M}$ acting on the manifold $M$ transitively.
So, we study homogeneous geodesics that are orbits of one-parametric subgroups of the group $G$.
We denote by $K = \{g \in G \, | \, go = o\}$ the isotropy subgroup of a point $o \in M$.
Thus, we can consider $M = G/K$ as a homogeneous space.
It is well known (see for example~\cite{kowalski-szenthe}) that there exists \emph{a reductive decomposition}
$\g = \m \oplus \kk$, where the subspace $\m$ is an $\Ad{K}$-invariant complement to the subalgebra $\kk$.
This subspace $\m$ is naturally isomorphic to the tangent space $T_oM$.

Also we assume that the group $G$ acts on the manifold $M$ \emph{effectively}, i.e.,
there are no nontrivial group elements that fix all points of the manifold $M$.

The author is grateful to prof.~Yu.\,L.~Sachkov for several useful remarks.

\section{\label{sec-geodesics}Sub-Riemannian geodesics}

Here we describe sub-Riemannian geodesics in terms of Hamiltonian vector fields on the cotangent bundle $T^*G$ via Pontryagin maximum principle.
Simultaneously we introduce some necessary notation.

Fist of all lift our sub-Riemannian problem from the homogeneous space $M = G/K$ to the Lie group $G$.
As a result we have an optimal control problem to find an optimal admissible curve $g: [0, T] \rightarrow G$ going from one left coset to another, see below~\eqref{eq-controlproblem}.
Due to the left-invariance we may assume that the starting left coset is just the subgroup $K$.
So, the problem is to find a control $u \in  L^{\infty}([0, T], \Delta)$ and a Lipschitz curve $g_u : [0, T] \rightarrow G$
for a given $g_1 \in G$ such that
\begin{equation}
\label{eq-controlproblem}
g_u(0) \in K, \quad g_u(T) \in g_1K, \quad \dot{g}_u(t) = L_{g_u(t)} u(t), \quad \int_{0}^{T}{\sqrt{\B(u(t), u(t))} \, dt} \rightarrow \min.
\end{equation}
It is equivalent to minimize the energy functional ${{1}\over{2}} \int_{0}^{T}{\B(u(t), u(t)) \, dt}$ instead of the functional~\eqref{eq-controlproblem}
because of Cauchy-Bunyakowski-Schwartz inequality. Below we use the energy functional.

\begin{remark}
\label{rem-G-left-K-right-invariance}
Notice that problem~\eqref{eq-controlproblem} is $G$-left and $K$-right invariant.
\end{remark}

Consider the following family of functions on the cotangent bundle $T^*G$ depending on the parameters $u \in \Delta$ and $\nu \in \R$:
$$
H_u^{\nu}(\lambda) = \langle L_{\pi(\lambda)} u, \lambda \rangle + {{\nu}\over{2}} \B(u, u), \qquad \lambda \in T^*G,
$$
where $\pi: T^*G \rightarrow G$ is the natural projection and
$\langle \,\cdot\, , \,\cdot\, \rangle$ is the natural pairing of elements from $T_gG$ and $T^*_gG$.

A necessary condition of optimality is given by the Pontryagin maximum principle~\cite{pontryagin,agrachev-sachkov}.

\begin{theorem}[Pontryagin maximum principle]
\label{th-pmp}
Assume that $\tilde{u} \in  L^{\infty}([0, T], \Delta)$ is an optimal control and
$\tilde{q} : [0, T] \rightarrow G$ is the corresponding optimal curve of the problem~$\eqref{eq-controlproblem}$, then
there exist a Lipschitz curve $\lambda: [0, T] \rightarrow T^*G$ and a number $\nu \leqslant 0$ such that\\
$(1)$ $(\nu, \lambda) \neq 0$,\\
$(2)$ $\pi(\lambda(t)) = \tilde{q}(t)$ for $t \in [0, T]$,\\
$(3)$ $\dot{\lambda}(t) = \vec{H}_{\tilde{u}(t)}^{\nu}(\lambda(t))$ for a.\,e. $t \in [0, T]$,\\
$(4)$ $($maximum condition$)$ $H_{\tilde{u}(t)}^{\nu}(\lambda(t)) = \max_{u \in \Delta}{H_u^{\nu}(\lambda(t))}$ for a.\,e. $t \in [0, T]$,\\
$(5)$ $($transversality condition$)$ $\lambda(t) \in (T_{\tilde{q}(t)} \tilde{q}(t)K)^{\circ}$ for $t \in [0, T]$.
\end{theorem}

\begin{definition}
A curve $\lambda: [0, T] \rightarrow T^*G$ satisfying the conditions of Theorem~\ref{th-pmp} is called \emph{an extremal}.
Its projection $\tilde{g} = \pi(\lambda)$ to the group $G$ is called \emph{an extremal curve}.
If the parameter $\nu \neq 0$, then corresponding extremals and extremal curves are called \emph{normal}.
If $\nu = 0$, then these curves are called \emph{abnormal}.
The further projection of a normal extremal curve to the homogeneous space $M = G/K$ is called \emph{a normal sub-Riemannian geodesic}.
\end{definition}

Below we will consider only normal geodesics.

Consider the left trivialization of the cotangent bundle
$$
\tau: T^*G \rightarrow \g^* \times G, \qquad \tau(\lambda) = (L_g^* \lambda, g), \qquad \lambda \in T_g^*G.
$$
By the transversality condition of the Pontryagin maximum principle we get
$$
p(t) = \tau(\lambda(t)) \in \kk^{\circ} \cong \m^* \quad \text{for an extremal} \quad \lambda : [0, T] \rightarrow T^*G.
$$
It is easy to see that the maximum condition in the normal case gives us
$$
\B(\tilde{u}(t), \,\cdot\, ) = L_{\tilde{g}(t)}^* \lambda(t) \in \kk^{\circ} \subset \g^*.
$$
It follows that $\tilde{u}(t) = \B^{-1}(p(t)|_{\Delta})$, where $\B : \Delta \rightarrow \Delta^*$ maps $X \in \Delta$ to $\B(X, \,\cdot\, ) \in \Delta^*$ and
by $p|_{\Delta}$ we denote the restriction of a linear function $p \in \g^*$ to the subspace $\Delta \subset \g$.
Thereby the left-invariant normal maximized Hamiltonian of the Pontryagin maximum principle reads as $H(p) = {{1}\over{2}}\B(p|_{\Delta}, p|_{\Delta})$
as a function of variable $p \in \kk^{\circ} \subset \g^*$. (Here we transfer the quadratic form $\B$ from $\Delta$ to $\Delta^*$.)

In the left trivialization the Hamiltonian system of the Pontryagin maximum principle
$\dot{\lambda}(t) = \vec{H}(\lambda(t))$ takes the form (see, for example, \cite{agrachev-sachkov})
\begin{equation}
\label{eq-hamiltoniansystem}
\begin{array}{lll}
\dot{p} & = & (\ad^*{d_pH})p,\\
\dot{g} & = & L_g d_pH, \\
\end{array}
\qquad
(p, g) = \tau(\lambda) \in \g^* \times G.
\end{equation}
Note that the first equation of system~\eqref{eq-hamiltoniansystem}, also called \emph{the vertical part of Pontryagin's Hamiltonian system},
does not depend on the second variable $g \in G$.
So, this subsystem can be studied independently.
Finally, note that every sub-Riemannian geodesic is determined by its initial momentum $p(0) = \lambda(0) \in \g^* = T^*_eG$.

\section{\label{sec-homogeneousgeodesics}Homogeneous geodesics}

Let $M$ be a sub-Riemannian manifold, and
let a Lie group $G \subset \Isom{M}$ be a subgroup of isometries that acts on the manifold $M$ transitively and effectively.
Therefore, $M = G/K$, where $K$ is an isotropy subgroup.

\begin{definition}
A sub-Riemannian geodesic $m : [0, T] \rightarrow M$ is called \emph{a homogeneous geodesic} if
it is an orbit of one-parametric subgroup of isometries, i.e., $m(t) = \exp{(tX)} m(0)$ for some element $X \in \g$ that is called
\emph{a geodesic vector}.
\end{definition}

\begin{remark}
An important property of homogeneous geodesics in view of optimal synthesis (or, generally speaking, in view of geometric control theory) is
\emph{an equioptimality}.
This means that the time of loss of optimality (\emph{the cut time}) for a geodesic does not depend on a starting point on this geodesic.
In other words, the cut time as a function of an initial momentum of a geodesic is constant on the trajectories of the vertical part of Pontryagin's Hamiltonian system.
Yu.\,L.~Sachkov~\cite{sachkov-problems} made an amazing observation that for all completely solved sub-Riemannian problems geodesics are equioptimal.
The reason of this phenomenon is unknown,
but homogeneous geodesics deliver a clear example of such situation.
At the same time, not every equioptimal geodesic is homogeneous~\cite{sachkov-equioptimal}.
\end{remark}

There is a well known criterion for homogeneous geodesics in the Riemannian case,
which appears first in works of B.~Kostant~\cite{kostant}, \'{E}.\,B.~Vinberg~\cite{vinberg-connections}
and later in paper of O.~Kowalski and L.~Vanhecke~\cite{kowalski-vanhecke}.

\begin{lemma}[Geodesic lemma]
\label{lem-gl}
A Riemannian geodesic passing through the point $o = eK$ with a tangent vector $X \in \m$ is homogeneous if and only if
there exists $Y \in \kk$ such that
$$
\B(X, [X+Y, \g]_{\m}) = 0,
$$
where the index $\m$ denotes the $\m$-component of a vector with respect to the reductive decomposition.
Corresponding one-parametric subgroup of isometries is $\{\exp{t(X+Y)} \, | \, t \in \R\}$.
\end{lemma}

The main difference of sub-Riemannian situation is that
a geodesic passing through a fixed point is determined by its initial momentum instead of an initial geodesic vector (tangent vector).
Therefore, we need a criterion in terms of initial momenta.

\begin{lemma}
\label{lem-sRgl}
The following conditions are equivalent for a sub-Riemannian structure.\\
$(1)$ A geodesic with an initial momentum $p \in \kk^{\circ}$ is homogeneous.\\
$(2)$ There exists a vector $X \in \g$ such that
$$
p([X, \g]) = 0 \qquad \text{and} \qquad X_{\m} = d_pH.
$$
$(3)$ The trajectory of the vertical part of the Pontryagin Hamiltonian system that corresponds to a geodesic lies in an $(\Ad^*{K})$-orbit on $\kk^{\circ}$.
\end{lemma}

\begin{remark}
\label{rem-fixedpoint}
In particular, any fixed point $p_0 \in \kk^{\circ}$ of the vertical part of the Hamiltonian vector field corresponds to a homogeneous geodesic, since $p_0 \in (\Ad^*{K})p_0$.
\end{remark}

\begin{proof}
(1)$\Rightarrow$(3)
Assume that a geodesic $m : [0, T] \rightarrow M$ with initial momentum $p(0) \in \kk^{\circ}$ is homogeneous, i.e., $m(t) = \exp{(tX)} m(0)$ for some $X \in \g$.
Consider the corresponding extremal $\lambda : [0, T] \rightarrow T^*G$.
Its points are elements of an orbit of the one-parametric subgroup shifted by some elements of the subgroup $K$ acting on the right.
Thus, there exists a curve $k : [0, T] \rightarrow K$ such that
$$
\lambda(t) = L_{\exp{(tX)}^{-1}}^* R_{k(t)^{-1}}^* p(0).
$$
Via left trivialization we have
$$
p(t) = L_{\exp{(tX)}k(t)}^* \lambda(t) = L_{k(t)}^* L_{\exp{(tX)}}^* L_{\exp{(tX)}^{-1}}^* R_{k(t)^{-1}}^* p(0) = (\Ad{k(t)})^* p(0).
$$
So, $p(t) = (\Ad^*{k(t)^{-1}}) p(0)$.
Therefore, the trajectory of the vertical subsystem $p : [0, T] \rightarrow \kk^{\circ}$ lies in an $\Ad^*{K}$-orbit.
\medskip

(3)$\Rightarrow$(2)
Assume that $p(t) = (\Ad^*{k(t)})p(0)$ for some curve $k : [0, T] \rightarrow K$.
Take a derivative $\frac{d}{dt}|_{t=0}$ of this equality.
We get $\dot{p}(0) = -(\ad^*{Z})p(0)$ where $Z = \dot{k}(0) \in \kk$.
But by the first equation of system~\eqref{eq-hamiltoniansystem} we have $\dot{p}(0) = (\ad^*{d_{p(0)}H})p(0)$.
Therefore, $\ad^*{(d_{p(0)}H+Z)} p(0) = p(0) ([X, \g]) = 0$, where $X = d_{p(0)}H + Z \in \g$ and $X_{\m} = d_{p(0)}H$.
\medskip

(2)$\Rightarrow$(1)
Assume that $p(0)([d_pH + Z, \g]) = 0$ for $Z \in \kk$.
This is equivalent to $\dot{p}(0) = - (\ad^*{Z})p(0)$.
Consider the curve $\lambda(t) = R_{\exp{(tZ)}}^* L_{\exp{t(d_pH + Z)}^{-1}}^* p(0) \in T^*G$.
We claim that $\dot{\lambda}(0)$ coincides with the Hamiltonian vector field at the point $\lambda(0)$.
Indeed, first consider the vertical component of the curve $\lambda$
$$
L_{\exp{t(d_pH + Z)}\exp{(-tZ)}}^* \lambda(t) = (\Ad^*{\exp{(tZ)}})p(0).
$$
Its tangent vector at the point $p(0)$ equals ${{d}\over{dt}}|_{t=0} (\Ad^*{\exp{(tZ)}})p(0) = - (\ad^*{Z}) p(0)$.
This is equal to $\dot{p}(0)$ by the assumption.
Second, the horizontal projection of the tangent vector to $\lambda(t)$ for $t=0$ is $d_pH \in \m$.
So, $\dot{\lambda}(0) = \vec{H}(\lambda(0))$.
Recall that the Hamiltonian vector field $\vec{H}$ is invariant under the left action of the group $G$ and
the right action of the group $K$ (Remark~\ref{rem-G-left-K-right-invariance}).
Thus, by definition of the curve $\lambda$, it is a trajectory of the Hamiltonian vector field.
Then the corresponding sub-Riemannian geodesic, i.e., the projection of $\lambda$ to the homogeneous space $G/K$
is an orbit of the one-parametric subgroup $\{\exp{t(d_pH+Z)}\} \subset G$.
\end{proof}

\begin{remark}
\label{rem-toth}
Condition~(2) of Lemma~\ref{lem-sRgl} in other terms first appears in G.\,Z.~T\'{o}th's paper~\cite{toth}.
\end{remark}

\begin{definition}
\label{def-go}
A sub-Riemannian manifold is called \emph{geodesic orbit} if every normal geodesic is an orbit of a one-parametric subgroup of isometries.
\end{definition}

Below we consider several examples of homogeneous geodesics and geodesic orbit Riemannian and sub-Riemannian manifolds.

\begin{example}
\label{ex-compact}
Consider a left-invariant Riemannian structure on a compact Lie group $M$ defined by the Killing form, i.e.,
$\Delta = \m$ and $\B = -\K_{\m}$.
Since the Killing form is left- and right-invariant,
then our Riemannian manifold is $(M\times M)/M$ as a homogeneous space.
Every geodesic is homogeneous by Lemma~\ref{lem-gl}, because for any $X \in \m$ the operator $\ad{X}$ is skew-symmetric.
In Hamiltonian terms of Lemma~\ref{lem-sRgl} the vertical part of the corresponding Pontryagin Hamiltonian vector field is trivial.
\end{example}

\begin{example}
\label{ex-Raxisym}
Consider the Lie group $M$ of proper isometries of a sphere (or a hyperbolic plane).
This group is isomorphic to the Lie group $\SO_3$ or $\PSL_2(\R)$, respectively.
Any left-invariant Riemannian metric is defined by a positive definite quadratic form $\B$ on $\m$.
Let $I_1, I_2, I_3 > 0$ be its eigenvalues and $X_1, X_2, X_3$ be corresponding eigenvectors.
The vertical part of the Pontryagin Hamiltonian system is
just Euler's equation for momenta of a rigid body with a fixed point.

If $I_1 = I_2 \neq I_3$, then the corresponding metric is called \emph{axisymmetric}.
In this case $G = \Isom{M} = M \times \SO_2$, and as a homogeneous space $M = (M \times \SO_2)/\SO_2$,
where the isotropy subgroup is an anti-diagonal subgroup $\SO_2 \subset M \times \SO_2$.
The solution of the vertical part of the Hamiltonian system is (see~\cite{podobryaev-sachkov-so3,podobryaev-sachkov-sl2})
\begin{equation}
\label{eq-axisym-vertpart}
p(t) = (\Ad^*{\kappa  p_3 X_3}) p(0),
\end{equation}
where $p_3$ is the third component of the covector $p$ in the basis dual to the basis $X_1, X_2, X_3$ and
$\kappa$ is a constant depending on the eigenvalues of the form $\B$.
Thus, by Lemma~\ref{lem-sRgl} all geodesics are homogeneous.
See Figure~\ref{pic-axisym} for the phase portrait of the vertical part of the Hamiltonian system on the level surface of the Hamiltonian $H = {{1}\over{2}}$.

\begin{figure}[h]
     \begin{minipage}[h]{0.45\linewidth}
        \center{\includegraphics[width=1\linewidth]{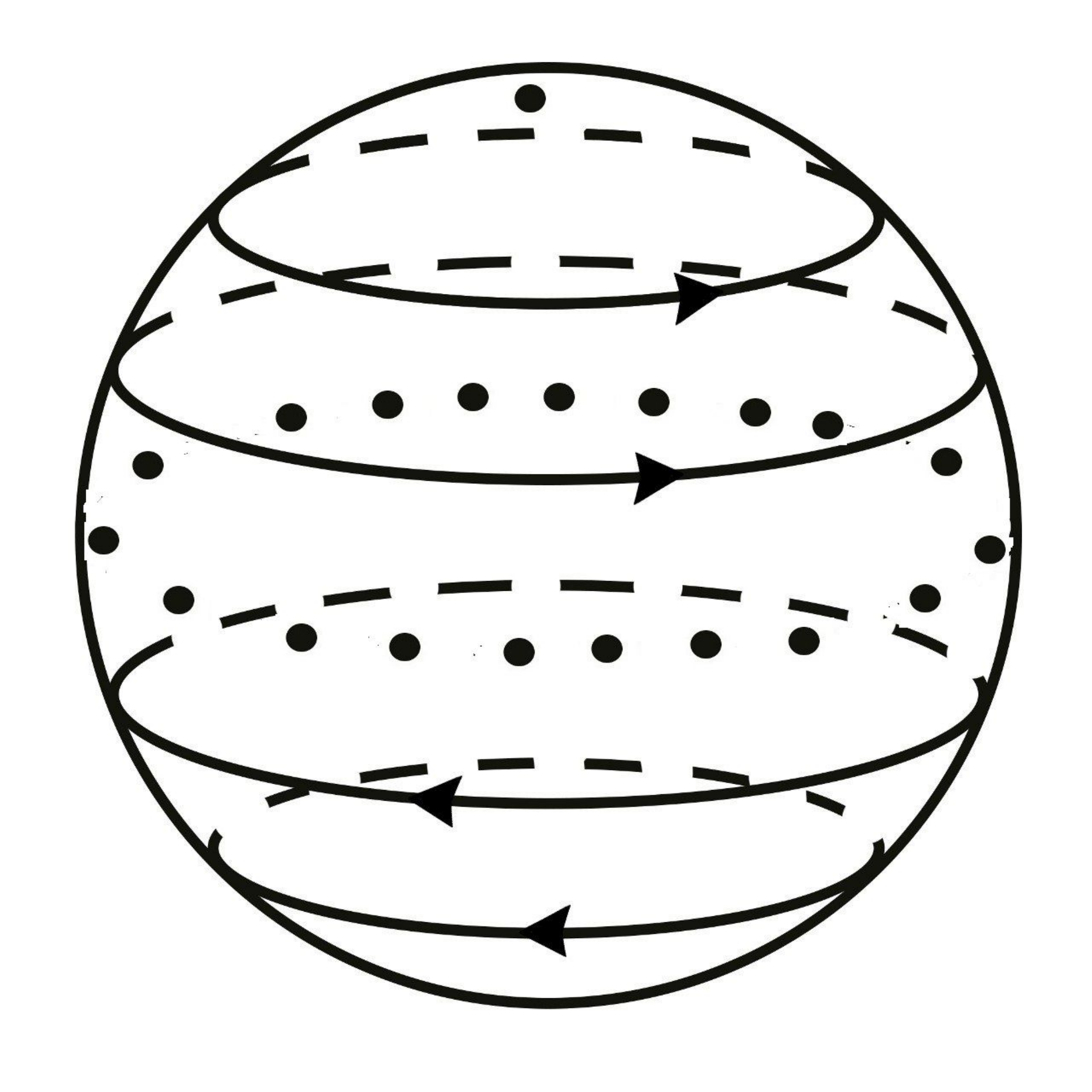}}
        \caption{\label{pic-axisym}Vertical subsystem in axisymmetric case.}
     \end{minipage}
     \hfill
     \begin{minipage}[h]{0.45\linewidth}
        \center{\includegraphics[width=1\linewidth]{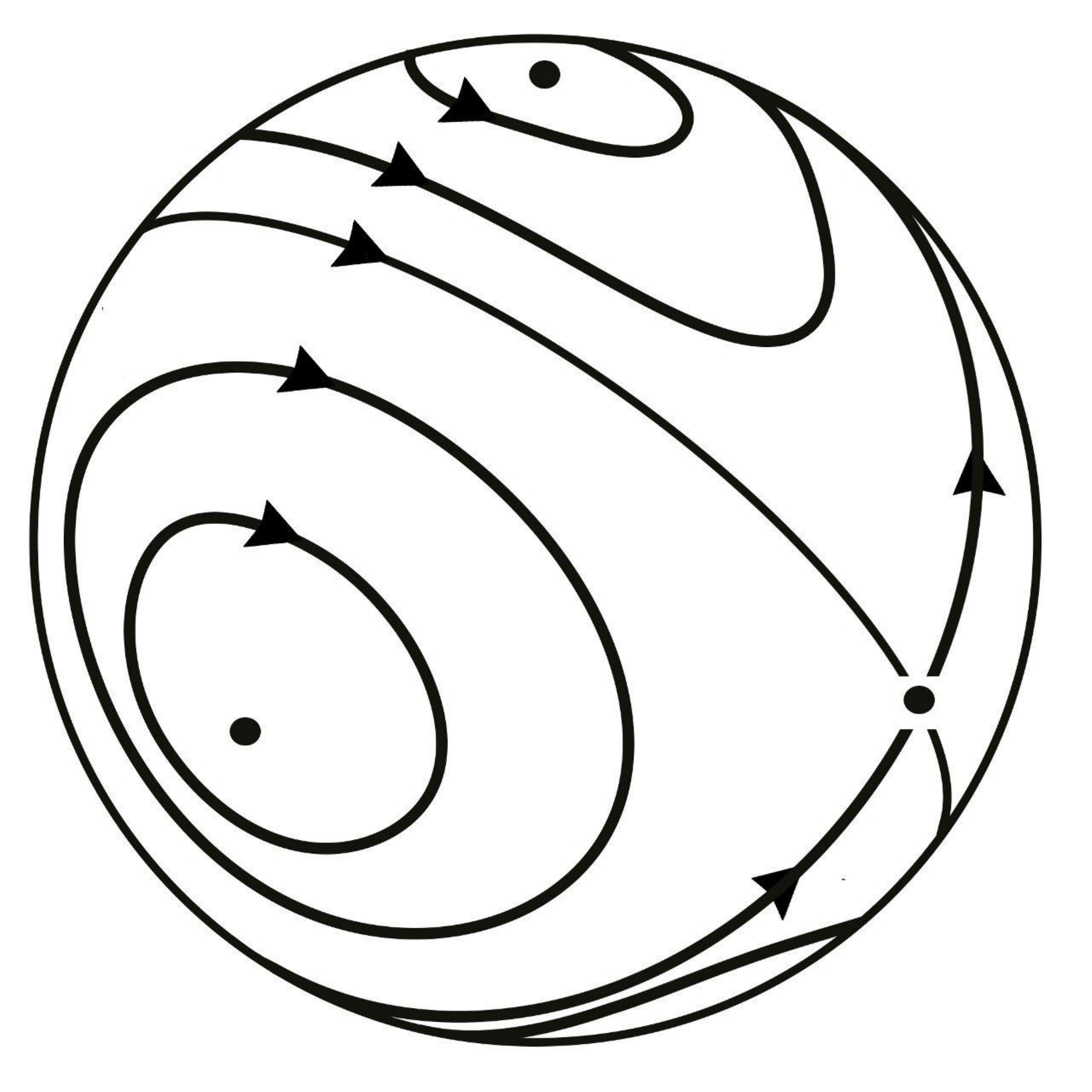}}
        \caption{\label{pic-general}Vertical subsystem in general case.}
     \end{minipage}
\end{figure}

In the general case of all different eigenvalues $I_1, I_2, I_3$ the isometry group is $M \times \Z_3$.
The vertical part of the Hamiltonian system has 6 fixed points (4 stable and 2 unstable ones), see Figure~\ref{pic-general}.
These points are the initial momenta of 6 homogeneous geodesics passing through the point $e$.
\end{example}

\begin{example}
\label{ex-sRaxisym}
Consider now a left-invariant sub-Riemannian structure on the group of proper isometries of a sphere (or a hyperbolic plane), i.e.,
$\SO_3$ or $\PSL_2(\R)$.
The sub-Riemannian distribution is $\Delta = \sspan{\{X_1, X_2\}}$,
where the restriction on $\Delta$ of the Killing form is positive defined in the case of $\PSL_2(\R)$,
and the form $\B$ on $\Delta$ is axisymmetric.
So, we have the sub-Riemannian structures on the homogeneous spaces (as in Example~\ref{ex-Raxisym})
\begin{equation}
\label{eq-weaklysym}
(\SO_3 \times \SO_2) / \SO_2, \qquad (\PSL_2(\R) \times \SO_2) / \SO_2.
\end{equation}
The shortest arcs of these sub-Riemannian structures were described by
V.\,N.~Berestovskii and I.\,A.~Zubareva~\cite{berestovskii-zubareva-so3,berestovskii-psl2}
and earlier U.~Boscain and F.~Rossi~\cite{boscain-rossi} found the cut loci.

The solution of the vertical part of the Pontryagin Hamiltonian system has the same form as in Example~\ref{ex-Raxisym}, see~\eqref{eq-axisym-vertpart}.
The corresponding phase portrait is similar to Figure~\ref{pic-axisym} except that
the level surface of the Hamiltonian is a cylinder instead of a sphere.
Hence, every geodesic is homogeneous.
Note, that it was shown by explicit formulae in papers~\cite{berestovskii-zubareva-so3,berestovskii-psl2}.
\end{example}

\begin{remark}
\label{rem-weaklysym}
The homogeneous spaces~\eqref{eq-weaklysym} are \emph{weakly symmetric}.
This means that for every pair of points there exists an isometry that swaps these points.
This notion was introduced by A.~Selberg and the second space of~\eqref{eq-weaklysym} is his original example.
It is known that Riemannian weakly symmetric spaces are geodesic orbit,
this is a result of J.~Berndt, O.~Kowalski and L.~Vanhecke~\cite{berndt-kowalski-vanhecke},
later V.\,N.~Berestovskii and Yu.\,G.~Nikonorov proved it using another method~\cite{berestovskii-nikonorov}.
For sub-Riemannian spaces this question was stated in paper~\cite{berestovskii-zubareva-sl2},
we give an answer in Section~\ref{sec-go}.
\end{remark}

\begin{remark}
\label{rem-kp}
The second part of Example~\ref{ex-sRaxisym} is a particular case of a left-invariant sub-Riemannian structure defined by the Cartan decomposition of the Lie algebra $\m$
$$
\m = \kk \oplus \pp, \quad \text{where} \quad [\kk, \kk] \subset \kk, \quad [\kk, \pp] \subset \pp, \qquad [\pp, \pp] \subset \kk.
$$
The sub-Riemannian structure is defined by the distribution $\Delta = \pp$ and the restriction of the Killing form to $\Delta$.
A.\,A.~Agrachev~\cite{agrachev-kp}, R.\,W.~Brockett~\cite{brockett} and V.~Jurdjevic~\cite{jurdjevic} noted that geodesics of this structure are products of two one-parametric subgroups
\begin{equation}
\label{eq-productof2one-parametricsubgroups}
\exp{t(X+Z)} \exp{(-tZ)}, \quad \text{where} \quad X \in \pp, \quad Z \in \kk.
\end{equation}
This means that this manifold is geodesic orbit, since the isometry group is $G = M \times K$,
and~\eqref{eq-productof2one-parametricsubgroups} is an orbit of a one-parametric subgroup of $G$.
\end{remark}

\begin{example}
Consider the optimal control problem for rolling of a sphere on a plane with twisting but without slipping.
I.\,Yu.~Beschastnyi~\cite{beschastnyi} studied it as a sub-Riemannian problem on the Lie group $M = \SO_3 \times \R^2$.

Let $V_1, V_2, V_3$ be an orthonormal basis of the Lie algebra $\so_3$ with respect to the Killing form,
and $e_1, e_2$ be a basis of the plane $\R^2$.
The distribution is defined by the orthonormal frame $e_1 - V_2, e_2 + V_1, V_3$.
Below we will identify the Lie algebra $\m$ with the dual space $\m^*$ via the basis $V_1, V_2, V_3, e_1, e_2$.
The vertical subsystem of Pontryagin's Hamiltonian system is as follows~\cite{beschastnyi}:
\begin{equation}
\label{eq-verticalballplane}
\left\{
\begin{array}{lll}
\dot{\varrho} & = & [\hat{p}, \varrho], \\
\dot{p} & = & 0, \\
\end{array}
\right.
\qquad
\begin{array}{l}
\varrho \in \so_3, \\
p \in \R^2, \\
\end{array}
\end{equation}
where $\hat{p} \in \so_3$ is the tangent vector to the one-parametric group of rotations around the vector $(p_1, p_2, 0) \in \R^3$ (here $p_1, p_2$ are the components of the vector $p$).

I.\,Yu.~Beschastnyi~\cite{beschastnyi} obtained geodesics as products of two one-parametric subgroups
$$
m(t) = \exp{t(\hat{p} + \varrho(0))} \exp{(-t\hat{p})}.
$$
But despite this, and also despite the fact that
the rotations $\exp{t\hat{\varrho}}$ are symmetries of system~\eqref{eq-verticalballplane},
this sub-Riemannian structure is not geodesic orbit.

Let us prove this by contradiction.
Assume that there is an isotropy group $K$ acting on $\m$ by conjugation.
By Lemma~\ref{lem-sRgl}, since any automorphism of the Lie algebra $\so_3$ is inner,
there exists $X \in \m$ such that $\dot{\varrho} = [X, \varrho]$ and $\dot{\hat{p}} = [X, \hat{p}]$.

But from~\eqref{eq-verticalballplane} we know that $\dot{\varrho} = [\hat{p}, \varrho]$ and $\dot{p} = 0$.
Thus, there are three possible cases.

1. If $p = 0$, then we can take $X = 0$.

2. If $p \neq 0$ and $\varrho$, $\hat{p}$ are collinear, then we can take $X = 0$ as well.

3. If $p \neq 0$ and $\varrho$, $\hat{p}$ are not collinear, then $X = \hat{p}$.

In the first two cases the geodesic with the initial momentum $(\varrho, p)$ is homogeneous,
since vertical subsystem~\eqref{eq-verticalballplane} is trivial.
In the third case we get that $\exp{X}$ is a rotation around a horizontal axis.
Hence, $\exp{X}$ does not preserve the distribution $\Delta$ in contradiction with the assumption that $\exp{X}$ is an isotropy.

Notice that in the first case the corresponding geodesic is a rotation of the sphere around vertical axis at a fixed point.
The second case corresponds to a rolling of the sphere along a straight line
(this is a non strict abnormal geodesic~\cite{beschastnyi}).
\end{example}

\section{\label{sec-nilpotent}Nilpotent geodesic orbit sub-Riemannian manifolds}

Nilpotent Lie groups play an important role in sub-Riemannian geometry because of existence of nilpotent approximation~\cite{agrachev-sarychev}.
In this section we show that left-invariant sub-Riemannian structures on the two step free nilpotent Lie groups are geodesic orbit
and prove that any left-invariant sub-Riemannian structure on a Carnot group of step more than $2$ cannot be geodesic orbit.

\begin{definition}
\label{def-Carnot}
A Lie algebra $\g$ is called \emph{a Carnot algebra of step $s$} if there is a decomposition
$$
\g = \g_1 \oplus \dots \oplus \g_s, \qquad [\g_1, \g_i] = \g_{i+1} \quad \text{for} \quad i=1,\dots,s, \qquad [\g_1, \g_s] = 0,
$$
such that $\g_1$ generates the Lie algebra $\g$.
The corresponding connected and simply connected nilpotent Lie group is called \emph{a Carnot group}.
The left shifts of the subspace $\g_1$ give a non-holonomic distribution on a Carnot group.
Any scalar product on $\g_1$ gives rise to a sub-Riemannian structure.
\emph{The rank of sub-Riemannian structure on a Carnot group} is equal to the dimension of the subspace $\g_1$.
\end{definition}

\begin{proposition}
\label{prop-step2}
A left-invariant sub-Riemannian structure on a Carnot group of step $2$ is geodesic orbit.
\end{proposition}

\begin{proof}
Consider the case of a free Carnot group of step $2$ and rank $r$ and the following model of such group.
Let $M = V \times \Lambda^2 V$, where $\dim{V} = r$.
Endow the set $M$ with a product rule as follows:
$$
(x_1, \omega_1) \cdot (x_2, \omega_2) = (x_1 + x_2, \, \omega_1 + \omega_2 + x_1 \wedge x_2),
\quad \text{for} \quad
x_1, x_2 \in V, \quad \omega_1, \omega_2 \in \Lambda^2 V.
$$
The corresponding tangent algebra is $\m = V \oplus \Lambda^2 V$ and the sub-Riemannian distribution is defined by the subspace $\Delta = V \oplus 0$.

The vertical part of the Pontryagin Hamiltonian system for normal geodesics reads as (see~\cite{rizzi-serres})
$$
\dot{p} = \varrho p, \quad \dot{\varrho} = 0, \quad \text{where} \quad
(p,\varrho) \in \m^* = V^* \oplus \Lambda^2 V^* = V^* \oplus \so(V).
$$

V.~Kivioja and E.~Le\,Donne~\cite{kivioja-ledonne} proved that $\Isom{M} = M \leftthreetimes \SO(V)$,
where the group $M$ acts by left shifts and the group $\SO(V)$ acts on $V$ tautologically and on $\Lambda^2 V \cong \so(V)$ by conjugations.
The isotropy subgroup of the identity point $(0, 0) \in M$ is $\SO(V)$.
The coadjoint action of $\SO(V)$ on $\m$ is tautological on the first component and conjugate on the second one
(due to the isomorphism of $V$ and $V^*$ given by the scalar product of the sub-Riemannian structure).
It is easy to see that the trajectories of the vertical part of the Hamiltonian system are tangent to the orbits of the group $\SO(V)$ in $\m^*$,
i.e., for any $(p,\varrho) \in \m^*$ there exists $X \in \so(V)$ such that
\begin{align*}
& \dot{p} = \varrho p = [X,\varrho],\\
& \dot{\varrho} = 0 = X\varrho.\\
\end{align*}
Indeed, we can take $X = \varrho$.
So, by Lemma~\ref{lem-sRgl} any normal geodesic of a free two-step Carnot group is homogeneous.

Any two-step nilpotent Lie algebra $\widehat{\g}$ is a factor of a free nilpotent Lie algebra $\g$ of the same rank
by an ideal $\ii \vartriangleleft \g$ laying at the second layer of the Lie algebra $\ii \subset \g_2$.
Since $\widehat{\g}^* \cong \ii^{\circ} \subset \g^*$,
we obtain that any geodesic $\widehat{\gamma}$ on the Carnot group $\widehat{G}$ can be lifted to
a geodesic $\gamma$ on the free Carnot group $G$ with an initial momentum from the subspace $\ii^{\circ}$.
It follows from Lemma~\ref{lem-sRgl} that if this lifted geodesic $\gamma$ is homogeneous,
then the geodesic $\widehat{\gamma}$ is homogeneous as well.
\end{proof}

\begin{remark}
\label{rem-step2}
The sub-Riemannian problems for free Carnot groups of step $2$ and small ranks are completely solved.
For rank $2$ the corresponding sub-Riemannian problem is equivalent to the classical variational Dido problem~\cite{sachkov-textbook},
the complete discovering of this sub-Riemannian structure was made by A.\,M.~Vershik and V.\,Ya.~Gershkovich~\cite{vershik-gershkovich}.
For rank $3$ O.~Myasnichenko~\cite{myasnichenko} constructed the cut locus and
A.~Montanari, D.~Morbidelli~\cite{montanari-morbidelli} derived the cut time and described singularities of the distance function.
L.~Rizzi, U.~Serres obtained some upper bounds for the cut time for ranks more than $3$~\cite{rizzi-serres}.
H.-Q.~Li and Ye~Zhang in recent preprints~\cite{li,li-zhang} introduced formulae for sub-Riemannian distance for Heisenberg type groups.
\end{remark}

\begin{remark}
\label{rem-Euclid}
Of course, Carnot groups of step $1$ (i.e., commutative groups) are geodesic orbit since the corresponding sub-Riemannian geometry is just Euclidian geometry.
\end{remark}

It turns out that sub-Riemannian structures on Carnot groups of step more than $2$ cannot be geodesic orbit.
This is a generalization of C.\,S.~Gordon's result~\cite{gordon} obtained for nilpotent Riemannian manifolds.

\begin{theorem}
\label{th-gonilp}
If a left-invariant sub-Riemannian structure on a Carnot group is geodesic orbit,
then its step equals $1$ or $2$.
\end{theorem}

\begin{proof}
Let us begin with considering a more general situation.
Assume that $G$ is a group of isometries acting transitively on a sub-Riemannian geodesic orbit manifold $M$ and
$K$ is an isotropy subgroup.
Consider a reductive decomposition $\g = \m \oplus \kk$.
The sub-Riemannian distribution is generated by the subspace $\Delta \subset \m$.
This subspace is $\Ad{K}$-invariant.
Let $\Delta^{\bot} \subset \m$ be its $\Ad{K}$-invariant complement.
Continue the non-degenerate quadratic form $\B$ on $\Delta$ (that defines the sub-Riemannian structure)
to a non-degenerate quadratic form $\widehat{\B}$ on $\m$ such way that
$\m = \Delta \oplus \Delta^{\bot}$ is the orthogonal decomposition with respect to the form $\widehat{\B}$ and
the form $\widehat{\B}$ is $\Ad{K}$-invariant.

Following~\cite{gordon} we claim that for any $X \in \Delta$ the operator $\pi_{\Delta^{\bot}} (\ad{X})$ is skew-symmetric on the subspace $\Delta^{\bot}$
with respect to the form $\widehat{\B}$,
where $\pi_{\Delta^{\bot}}: \m \rightarrow \Delta^{\bot}$ is the orthogonal projection.

Identify the tangent space $\m$ with the dual space $\m^*$ by the map $\widehat{\B} : \m \rightarrow \m^*$, where $\m \ni \xi \mapsto \widehat{\B}(\xi, \,\cdot\,)$.
Since the sub-Riemannian structure is geodesic orbit, applying Lemma~\ref{lem-sRgl}, we get that
for any $p = \widehat{\B}(X + Y) \in \kk^{\circ}$, where $X \in \Delta$ and $Y \in \Delta^{\bot}$,
there exists $Z \in \kk$ such that
$$
p([d_pH + Z, \g]) = \widehat{\B}(X+Y, [d_pH + Z, \g]) = \widehat{\B}(X+Y, [X + Z, \g]) = 0,
$$
since $H(p) = {{1}\over{2}}\widehat{\B}(X,X)$ and $d_pH = X$ in the sub-Riemannian case.

This is equivalent to
$$
\widehat{\B}(X+Y, [X + Z, \g]) = \widehat{\B}(X, [X + Z, \g]) + \widehat{\B}(Y, [X + Z, \g]) = 0.
$$
In particular,
\begin{equation}
\label{eq-nilp}
\widehat{\B}(X, [X + Z, Y]) + \widehat{\B}(Y, [X + Z, Y]) = 0.
\end{equation}

Let us show that the first term of equation~\eqref{eq-nilp} equals zero.
Indeed, since our manifold is geodesic orbit, Lemma~\ref{lem-sRgl} is satisfied for $p = \widehat{\B}(X)$, i.e.,
there exists $Z_X \in \kk$ such that
$$
\widehat{\B}(X, [X + Z_X, \g]) = \widehat{\B}(X, [X + Z, \g]) + \widehat{\B}(X, [Z_X - Z, \g]) = 0.
$$
In particular,
\begin{equation}
\label{eq-nilpsec}
\widehat{\B}(X, [X + Z, Y]) + \widehat{\B}(X, [Z_X - Z, Y]) = 0.
\end{equation}
But $\m = \Delta \oplus \Delta^{\bot}$ is $\Ad{K}$-invariant decomposition.
It follows that $[Z_X - Z, Y] \in \Delta^{\bot}$ since $Z_X - Z \in \kk$.
So, the second term of equation~\eqref{eq-nilpsec} equals zero, since $X \in \Delta$ and
$\Delta$ is orthogonal to $\Delta^{\bot}$.
We obtain $\widehat{\B}(X, [X + Z, Y]) = 0$, i.e.,
the first term of equation~\eqref{eq-nilp} is zero.

It follows that the second term of equation~\eqref{eq-nilp} is zero as well:
\begin{equation}
\label{eq-nilpfinal}
\widehat{\B}(Y, [X + Z, Y]) = \widehat{\B}(Y, [X, Y]) + \widehat{\B}(Y, [Z, Y]) = 0.
\end{equation}
The second summand in equation~\eqref{eq-nilpfinal} equals zero, because the operator $\ad{Z}$ is skew-symmetric for $Z \in \kk$.
We get
$$
\widehat{\B}(Y, [X, Y]) = \widehat{\B}(Y, [X, Y]_{\Delta^{\bot}}) = \widehat{\B}(Y, \pi_{\Delta^{\bot}} (\ad{X})Y) = 0.
$$
This means that the operator $\pi_{\Delta^{\bot}} (\ad{X})$ is skew-symmetric as well.

Finally, by theorem of V.~Kivioja and E.~Le\,Donne~\cite{kivioja-ledonne} if $N$ is a nilpotent Lie group with a left-invariant sub-Riemannian structure,
then $G = N \leftthreetimes K$, where $K$ is an isotropy subgroup.
We have an $\Ad{K}$-invariant decomposition $\n = \Delta \oplus [\n, \n]$.
We proved that for any $X \in \Delta$ the operator $\pi_{[\n, \n]} (\ad{X})$ is skew-symmetric.
On the other hand, this operator is nilpotent.
It follows that it is equal to zero.
Since $\Delta$ generates the Lie algebra $\n$, we obtain $[\n, [\n, \n]] = 0$.
So, the step of our Carnot group is less than $3$.
\end{proof}

\begin{example}
\label{ex-cartan}
Consider \emph{the Cartan group}, i.e., the free Carnot group of rank~2 and step~3.
The corresponding Carnot algebra $\g = \g_1 \oplus \g_2 \oplus \g_3$,
where $\g_1 = \sspan{\{X_1, X_2\}}$, $\g_2 = \sspan{\{X_3\}}$, $\g_3 = \sspan{\{X_4, X_5\}}$,
has the following non-zero commutators of the basis elements:
$$
[X_1, X_2] = X_3, \qquad [X_1, X_3] = X_4, \qquad [X_2, X_3] = X_5.
$$
The isotropy subgroup of the group of isometries is $\SO_2$
acting tautologically on $\g_1$, $\g_3$ and trivially on $\g_2$~\cite{sachkov-didona}.
There are $3$ independent Casimir functions:
$$
\frac{1}{2}h_3^2 + h_1h_5 - h_2h_4, \qquad h_4, \qquad h_5,
$$
where $h_i = \langle X_i, \, \cdot \, \rangle$ for $i=1,\dots,5$ are linear functions on $\g^*$.
These functions together with the Hamiltonian $H = {{1}\over{2}}(h_1^2 + h_2^2)$ from a system of $4$ first integrals 
of the vector field $\vec{H}$ in the $5$-dimensional space $\g^*$.
Consider trajectories of the vertical part of the Hamiltonian vector field.
In the case of $h_4 = h_5 = 0$ the corresponding trajectories are the same as for \emph{the Heisenberg group},
i.e., the free Carnot group of rank~2 and step~2.
In the case $h_4^2 + h_5^2 \neq 0$ the corresponding trajectories are intersections of two quadrics.
Hence, by Lemma~\ref{lem-sRgl} the only homogeneous geodesics are geodesics with initial momenta such that $h_4 = h_5 = 0$.
\end{example}

\begin{remark}
\label{rem-hg-on-carnot-groups}
It would be interesting to find all the homogeneous geodesics for Carnot groups.
We can just say that there are some nontrivial homogeneous geodesics, i.e., geodesics corresponding to nontrivial trajectories of the vertical part of the Hamiltonian vector field $\Hv$.
Note, that for a free Carnot group the isotropy subgroup contains $\SO(\g_1)$.
Indeed, the elements of the group $\SO(\g_1)$ are symmetries of the vector field $\Hv$ on $\g^*$,
these symmetries induce symmetries of the exponential map that are isometries~\cite{podobryaev-sym}.
For a free Carnot group there exist homogeneous geodesics that appear from the free Carnot group of the same rank and lower step due to the natural factorization. Besides this, for step~3 free Carnot groups there exist series of two-dimensional coadjoint orbits
that are organized as coadjoint orbits of the Heisenberg group~\cite{podobryaev2}.
So, geodesics with initial momenta that lie in this kind of coadjoint orbits are homogeneous by Lemma~\ref{lem-sRgl}.
\end{remark}

\section{\label{sec-go}Some general properties of\\ geodesic orbit sub-Riemannian manifolds}

Here we prove some general facts about geodesic orbit sub-Riemannian manifolds and obtain a series of important examples.
First we get the following criterion.

\begin{proposition}
\label{prop-gocriterion}
The homogeneous sub-Riemannian space $M = G/K$ is geodesic orbit if and only if
$$
\{H, \R[\m^*]^K\} = 0,
$$
where $H$ is the normal Pontryagin Hamiltonian,
$\R[\m^*]^K$ is the algebra of left-invariant polynomial functions on $T^*M$
\emph{(}i.e., the algebra of the left shifts of $\Ad^*{K}$-invariant polynomial functions on $\m^*$\emph{)} and
$\{\,\cdot\, , \,\cdot\, \}$ is the canonical Poisson structure on $T^*M$.
\end{proposition}

\begin{proof}
Let $F \in \R[\m^*]^K$ be a left-invariant function on $T^*M$.
It is well known that
$$
\{H, F\}|_{\m^*} = \{H|_{\m^*}, F|_{\m^*}\} = \vec{H}_{\mathrm{vert}}(F|_{\m^*}),
$$
where the second bracket is the canonical Poisson bracket on $\g^* \supset \kk^{\circ} \cong \m^*$ and
$\vec{H}_{\mathrm{vert}}$ is the vertical part of the Hamiltonian vector field that coincides
with the Hamiltonian vector field of the function $H|_{\m^*}$ on $\m^*$.
Thus, $\{H, F\} = 0$ if and only if $\vec{H}_{\mathrm{vert}}(F|_{\m^*}) = 0$.

It follows from Lemma~\ref{lem-sRgl} that the sub-Riemannian manifold $M$ is geodesic orbit if and only if
any trajectory of the vertical part of the normal Pontryagin Hamiltonian system is tangent to $\Ad^*{K}$-orbits in $\m^*$.
It is equivalent to the following fact:
any function that is constant on $\Ad^*{K}$-orbits has zero derivative in the direction of $\vec{H}_{\mathrm{vert}}$.
\end{proof}

From Proposition~\ref{prop-gocriterion} it follows that the geodesic flow of a geodesic orbit sub-Riemannian manifold is integrable in noncommutative sense.
Recall some necessary definitions.

\begin{definition}
\label{def-completealgebra}
An algebra of functions $\mathcal{F}$ on the cotangent bundle $T^*M$ is called \emph{complete} if
$$
\dim{\sspan{\{d_xf \, | \, f \in \mathcal{F}\}}} + \dim{\Ker{\{\, \cdot \, , \, \cdot \,\}|_{\mathcal{F}}}} = \dim{T^* M}.
$$
\end{definition}

\begin{remark}
\label{rem-completealgebra}
Roughly speaking, this means that there is a set of $\dim{M} + k$ independent functions on $T^*M$,
but instead of Liouville integrable condition we have that the matrix of the Poisson brackets of these function is nonzero and has rank $2k$.
\end{remark}

Recall the definition of the momentum map and some of its properties.

\begin{definition}
\label{def-momentummap}
For $\xi \in \g$ take the corresponding velocity field $H_{\xi}$ of the action of the group $G$ on the manifold $M$:
$$
H_{\xi}(m) = \frac{d}{dt}\big|_{t=0} L_{g(t)} m \in T_mM, \qquad m \in M,
$$
where $g(t) \in G$ is the one-parametric subgroup such that $\dot{g}(0) = \xi$.
Consider this velocity field $H_{\xi}$ as a linear function on fibers of the cotangent bundle $T^*M$.
Define \emph{the momentum map}
$$
\mu : T^*M \rightarrow \g^*, \qquad \mu(\lambda)(\xi) = H_{\xi}(\lambda), \quad \text{for} \quad \lambda \in T^*M, \ \xi \in \g.
$$
\end{definition}

\begin{remark}
\label{rem-symplecticorthogonal}
Notice that $\vec{H}_{\xi}$ is a velocity field $\xi_*$ of the $G$-action on $T^*M$:
$$
\xi_*(\lambda) = \frac{d}{dt}\big|_{t=0} L_{g(t)^{-1}}^* \lambda \in T_{\lambda}T^*M, \qquad \lambda \in T^*M.
$$
By definition of a symplectic gradient we obtain the following expression for the differential of the momentum map
$$
d\mu(\xi) = dH_{\xi} = \omega(\vec{H}_{\xi}, \,\cdot\,) = \omega(\xi_*, \,\cdot\,) \quad \text{for} \quad \xi \in \g,
$$
where $\omega$ is the canonical symplectic structure on $T^*M$.
This means that the orbits of the left $G$-action on $T^*M$ and
the fibers of the momentum map are orthogonal with respect to the symplectic structure,
for details see, for example, \cite{marsden-montgomery-ratiu} or \cite{vinberg-weaklysym}.
\end{remark}

\begin{remark}
\label{rem-leftright}
For any left-invariant function $F$ on $T^*M$ and for any $\xi \in \g$ by Definition~\ref{def-momentummap} we have
$$
\xi_* \cdot F = \{H_{\xi}, F\} = \{\mu^*\xi, F\} = 0.
$$
Hence, $F$ commutes with pullback with respect of momentum map of any linear function on $\g^*$.
Then, $\{F, \mu^*\R[\g^*]\} = 0$.
\end{remark}

\begin{example}
\label{ex-completealgebra}
The algebra $\R[\m^*]^K + \mu^*(\R[\g^*])$ is complete.
This follows immediately from Remark~\ref{rem-symplecticorthogonal}.
\end{example}

\begin{definition}
\label{def-intnoncom}
A Hamiltonian vector field $\vec{H}$ on the cotangent bundle $T^*M$ is called \emph{integrable in noncommutative sense} if
there exists a complete algebra $\mathcal{F}$ of functions on $T^*M$ that are first integrals of the vector field $\vec{H}$.
In other words, $\{H, \mathcal{F}\} = 0$.
\end{definition}

\begin{theorem}
\label{th-geodesicflow}
Consider a homogeneous space $M = G/K$.
If any trajectory of a left-invariant Hamiltonian vector field $\vec{H}$ on the cotangent bundle $T^*M$ is homogeneous,
then this Hamiltonian vector field is integrable in noncommutative sense.
In particular, the geodesic flow of a geodesic orbit sub-Riemannian manifold is integrable in noncommutative sense.
\end{theorem}

\begin{proof}
Take the complete algebra from Example~\ref{ex-completealgebra}, i.e., $\mathcal{F} = \R[\m^*]^K + \mu^*(\R[\g^*])$.
Any trajectory of the Hamiltonian vector field $\vec{H}$ is homogeneous, so, by Proposition~\ref{prop-gocriterion} we have
$\{H, \R[\m^*]^K\} = 0$.

Since $H$ is a left-invariant function on $T^*M$,
then by Remark~\ref{rem-leftright} we have $\{H, \mu^*(\R[\g^*]\} = 0$.

So, we obtain that $\{H, \mathcal{F}\} = 0$ for the complete algebra $\mathcal{F}$.
\end{proof}

\begin{remark}
\label{rem-Riemanniangoflow}
In the case of Riemannian geodesic orbit homogeneous spaces this is an observation of B.~Jovanovi\'{c}~\cite{jovanovic}.
\end{remark}

\begin{remark}
\label{rem-interation}
This result is consistent with the fact that the geodesic flow
for sub-Riemannian structures on $2$-step Carnot groups is integrable in trigonometric functions~\cite{vershik-gershkovich,myasnichenko,rizzi-serres}.
Recall that these sub-Riemannian structures are geodesic orbit, see Proposition~\ref{prop-step2}.
But sub-Riemannian structures on Carnot groups of step more than $2$ cannot be geodesic orbit.
Actually, it was shown numerically in paper~\cite{bizyaev-borisov-kilin-mamaev} that the geodesic flow of sub-Riemannian structures on $3$-step Carnot groups of rank more than $2$ is not Liouville integrable and as a consequence is not integrable in noncommutative sense.
(For $3$-step Carnot group of rank $2$, i.e., \emph{the Cartan group}, the sub-Riemannian geodesic flow is integrable in elliptic functions~\cite{sachkov-didona}.)
The sub-Riemannian geodesic flow for Carnot groups of step more than $3$ is not Liouville integrable as it was proved by L.\,V.~Lokutsievskii and Yu.\,L.~Sachkov~\cite{lokutsievskii-sachkov}.

We know near to nothing about dynamics of sub-Riemannian geodesic flow on Carnot groups of step more than $2$,
except the Cartan group~\cite{sachkov-didona}.
There are only particular results on an asymptotic behaviour of extremal controls~\cite{podobryaev1,podobryaev2}.
\end{remark}

Now we will get a series of examples of geodesic orbit homogeneous spaces.

\begin{definition}
A homogeneous space $M = G/K$ is \emph{weakly commutative} if the algebra of left-invariant polynomial functions on $T^*M$ is commutative with respect to the Poisson bracket.
\end{definition}

\begin{theorem}
If a homogeneous sub-Riemannian manifold $M = G/K$, where $G \subset \Isom{M}$, is weakly commutative,
then it is geodesic orbit.
In particular, weakly symmetric sub-Riemann\-ian manifolds are geodesic orbit.
\end{theorem}

\begin{proof}
Since the manifold $M$ is weakly commutative and the normal left-invariant Pontryagin Hamiltonian $H$ is an $\Ad^*{K}$-invariant polynomial function on $\m^*$,
we obtain $\{H, \R[\m^*]^K\} = 0$.
So, by Proposition~\ref{prop-gocriterion} the sub-Riemannian manifold $M$ is geodesic orbit.

It is well known that weakly symmetric spaces (and in particular symmetric spaces) are weakly commutative~\cite{akhiezer-vinberg,vinberg-weaklysym}.
But there are no symmetric spaces with an invariant sub-Riemannian structure~\cite{berestovskii-intristic}.
\end{proof}

See Remark~\ref{rem-weaklysym} above for references on the Riemannian version of this theorem.

\section{\label{sec-existence}Existence of a homogeneous geodesic}

In this section we prove that under some broad conditions there exists a homogeneous geodesic passing through an arbitrary point of a homogeneous sub-Riemannian manifold.
Moreover, we prove that there exists a fixed point of the vertical part of the corresponding Hamiltonian system in this case.

\begin{theorem}
\label{th-existence}
Let $(M, \Delta, \B)$ be a sub-Riemannian manifold and
let $G \subset \Isom{M}$ be a connected subgroup of isometries acting on the manifold $M$ transitively and effectively.
Let $K \subset G$ be an isotropy subgroup and let $\g = \m \oplus \kk$ be a reductive decomposition.
Denote by $\K$ the Killing form on the Lie algebra $\g$.
If $\Ker{\K} = \m$ or $\K|_{\Delta} \neq 0$, then for any point of the manifold $M$ there exists a homogeneous sub-Riemannian geodesic passing through that point.
\end{theorem}

\begin{proof}
It is sufficient to show that there exists such geodesic passing through the point $o = e K$.

Denote by $\rr = \Ker{\K}$ the kernel of the Killing form.
Note that $\rr$ is a solvable ideal in the Lie algebra $\g$.
Since the Killing form is non degenerate on $\kk$ \cite{kowalski-szenthe}, we obtain $\rr \subset \m$.
There are two cases: $\rr = \m$ and $\rr \subsetneq \m$.

In the fist case following~\cite{kowalski-szenthe} we can assume that a solvable group of isometries acts on the manifold $M$ transitively.
Let us denote this group by the same letter $G$.
Then $[\g, \g] \neq \g$.
So, there exists $p \in \g^*$ such that $p([\g, \g]) = 0$ and the conditions of Lemma~\ref{lem-sRgl} are satisfied automatically.

Consider now the second case $\rr \subsetneq \m$.
The Lie group corresponding to the ideal $\rr$ acts on the manifold $M$.
We will reduce our sub-Riemannian structure to a sub-Riemannian structure on the factor manifold by this action (see Lemma~\ref{lem-factorization} below).
Any geodesic of reduced sub-Riemannian structure lifts to the geodesic of original sub-Riemannian structure.
Herewith homogeneous geodesics lift to homogeneous geodesics.

After that  in Lemma~\ref{lem-semisimple} we consider a semisimple group of isometries of the new sub-Riemannian structure such that
the restriction of the Killing form to the sub-Riemannian distribution is non trivial.
Then Theorem~\ref{th-existence} follows from Lemmas~\ref{lem-factorization}, \ref{lem-semisimple}.
\end{proof}

\begin{lemma}
\label{lem-factorization}
Let $\ii \vartriangleleft \g$ be an ideal such that it does not contain the distribution $\Delta$.
There exists an invariant sub-Riemannian structure on $G/\widehat{K}$ such that
any geodesic of this structure lifts to a geodesic of the sub-Riemannian structure on $G/K$,
where $\widehat{K} \subset G$ is a Lie subgroup with the tangent algebra $\kk+\ii$.
\end{lemma}

\begin{proof}
Consider the subspace $\ii_{\Delta} = \Delta \cap \ii$ and its orthogonal complement $\ii_{\Delta}^{\bot_{\B}}$ in $\Delta$ with respect to the form $\B$.
Let $\widehat{\Delta} = \pi(\ii_{\Delta}^{\bot_{\B}})$ be the image of this orthogonal complement under the factorization map $\pi : \g \rightarrow \g/\ii$.
Notice that the subspace $\widehat{\Delta}$ and the subalgebra $\pi(\kk)$ generate the Lie algebra $\g/\ii$ since $\ii$ is an ideal and $\Delta + \kk$ generates $\g$.
Moreover, there exists a scalar product on $\widehat{\Delta}$ that is inherited from the scalar product $\B$ on $\Delta$.
So, we can define a sub-Riemannian structure on $G/\widehat{K}$.

Any extremal of the sub-Riemannian problem on $G/\widehat{K}$ is an extremal of the sub-Riemannian problem on $G/K$
with an initial momentum from the subspace $(\kk + \Delta\cap\ii)^{\circ} \subset \kk^{\circ}$.
This means that the corresponding geodesic on $G/\widehat{K}$ lifts to a geodesic on $G/K$.
\end{proof}

\begin{remark}
\label{rem-chaplygin}
Note that in this case the sub-Riemannian structure is a kind of generalization of a Chaplygin system~\cite{alekseevsky}.
\end{remark}

\begin{lemma}
\label{lem-semisimple}
Assume that the restriction of the Killing form to the subspace $\Delta$ is non trivial.
Then there exists a homogeneous sub-Riemannian geodesic.
\end{lemma}

\begin{proof}
The kernel $\rr$ of the Killing form does not contain the subspace $\Delta$ since $\K|_{\Delta} \neq 0$.
According to Lemma~\ref{lem-factorization} we can assume that $\rr = 0$.

Denote by $\Delta^{\bot_{\K}}$ the orthogonal complement of $\Delta$ with respect to the Killing form.
Take the subspace $\Gamma \subset \Delta$ that is the orthogonal complement of $\Delta \cap \Delta^{\bot_{\K}}$ in $\Delta$ with respect to the form $\B$.
Obviously, $\Gamma^{\bot_{\K}} \supset \Delta^{\bot_{\K}}$.

Moreover, $\Gamma \cap \Gamma^{\bot_{\K}} = 0$.
Indeed, suppose by contradiction that there exists $\xi \in \Gamma \cap \Gamma^{\bot_{\K}}$ such that $\xi \neq 0$.
Then since $\xi \in \Delta$ we have $\K(\xi, \Delta \cap \Delta^{\bot_{\K}})$ and since $\xi \in \Gamma^{\bot_{\K}}$, we obtain $\K(\xi, \Gamma) = 0$.
So, by definition of $\Gamma$ we get $\K(\xi, \Delta) = 0$.
It follows that $\xi \in \Delta \cap \Delta^{\bot_{\K}}$ in contradiction with $\xi \in \Gamma$.
Next, since $\K$ is non degenerate on $\g$, we have $\g = \Gamma \oplus \Gamma^{\bot_{\K}}$.

Now construct a scalar product $\widehat{\B}$ on $\g$ such that $\widehat{\B}|_{\Delta} = \B_{\Delta}$ and
the subspaces $\Gamma$ and $\Gamma^{\bot_{\K}}$ are orthogonal with respect of $\widehat{\B}$.

According to the result of O.~Kowalski and J.~Szenthe~\cite{kowalski-szenthe} for the Riemannian metric that is defined by the scalar product $\widehat{\B}$
there exists a homogeneous geodesic.
We will show that the initial momentum of this geodesic is also an initial momentum of a homogeneous geodesic of our sub-Riemannian structure.
In other words, we will find a geodesic vector that is one of the main axes of the forms $\K$ and $\widehat{\B}$ at the same time and
we prove that this vector belongs to the subspace $\Gamma \subset \Delta$.

Following paper~\cite{kowalski-szenthe} consider the operator $A = \K^{-1}\widehat{\B} : \g \rightarrow \g$, where
$$
\begin{array}{lll}
\K : \g \rightarrow \g^*, & \quad & \g \ni \xi \mapsto \K(\xi,\,\cdot\,),\\
\widehat{\B} : \g \rightarrow \g^*, & \quad & \g \ni \xi \mapsto \widehat{\B}(\xi,\,\cdot\,).\\
\end{array}
$$
Note that $\Gamma$ is an invariant subspace of the operator $A$.
Indeed, $\widehat{\B}(\Gamma, \Gamma^{\bot_{\K}}) = 0$, in other words $\widehat{\B}(\Gamma) \subset (\Gamma^{\bot_{\K}})^{\circ} = \K(\Gamma)$,
this means that $\K^{-1}\widehat{\B}(\Gamma) \subset \Gamma$.

Take now an eigenvector $X$ of the operator $A$ in the subspace $\Gamma$ with non zero eigenvalue $\lambda \neq 0$.
Then for $p = \widehat{\B}(X)$ we have
$$
H(p) = {{1}\over{2}}\B(p|_{\Delta}, p|_{\Delta}), \qquad d_pH = \widehat{\B}^{-1}p = X.
$$
We obtain
$$
\begin{array}{lcl}
p([d_pH, \g]) & = & \widehat{\B}(X, [d_pH, \g]) = \widehat{\B}(X, [X, \g]) = \\
& = & \K(AX, [X, \g]) = {{1}\over{\lambda}}\K(X, [X, \g]) = {{1}\over{\lambda}}\K([X,X],\g) = 0.\\
\end{array}
$$
Therefore conditions of Lemma~\ref{lem-sRgl} are satisfied.
\end{proof}

From the proof of Theorem~\ref{th-existence} immediately follows

\begin{corollary}
\label{crl-solv-semisimple}
$(1)$ If a solvable group of isometries acts transitively on a sub-Riemannian manifold,
then there is a homogeneous geodesic passing through an arbitrary point.\\
$(2)$ If a semisimple group of isometries acts transitively on a sub-Riemannian manifold and
the restriction of the Killing form to the sub-Riemannian distribution is non trivial,
then there is a homogeneous geodesic passing through an arbitrary point.
\end{corollary}

\section*{\label{sec-conclusion}Conclusion}

We generalized several facts known about Riemannian homogeneous geodesics to the sub-Riemannian case.
We used the Hamiltonian approach to formulate an analogue of the Geodesic Lemma, i.e.,
a criterion for geodesic to be homogeneous.

We studied several properties of geodesic orbit sub-Riemannian manifolds.
In parti\-cu\-lar, the corresponding geodesic flow is integrable in non-commutative sense.
Since we obtain that Carnot groups of step more than $2$ cannot be geodesic orbit,
this agrees with previously known facts about the integrability of geodesic flows on Carnot groups.

Finally, we proved that there exists at least one homogeneous geodesic for a sub-Riemannian manifold with
a semisimple group of isometries such that the restriction of the Killing form to the distribution is non trivial.
However, it is still unknown if this condition is essential.
It seems that it is possible to avoid this condition using another method for proof.

Nevertheless, several important questions remain open.
How to describe all geodesic orbit sub-Riemannian manifolds?
How to describe all homogeneous geodesics of a given sub-Riemannian structure?
It requires the knowledge of the group of isometries, that is a separate problem.

\end{document}